\documentclass{amsart}
\usepackage{verbatim,amsmath,amssymb,amscd,color,graphics}
\usepackage{tikz}
\usepackage[vcentermath]{myyoungtab}

\usepackage[all]{xy} 
\CompileMatrices

\newcommand{\ab}{\ol{1}}
\newcommand{\bb}{\ol{2}}
\newcommand{\cb}{\ol{3}}
\newcommand{\db}{\ol{4}}
\newcommand{\bij}{\Phi}

\newcommand{\diagram}{\mathrm{diagram}}

\newcommand{\geh}{\mathfrak{g}}

\newcommand{\la}{\lambda}
\newcommand{\La}{\Lambda}
\newcommand{\lev}{\mathrm{lev}}
\newcommand{\mb}{\text{}}
\newcommand{\ol}{\overline}

\newcommand{\ve}{\varepsilon}
\newcommand{\vp}{\varphi}
\newcommand{\wt}{\mathrm{wt}\,}
\newcommand{\Z}{\mathbb{Z}}

\numberwithin{equation}{section}

\newtheorem{theorem}{Theorem}
\newtheorem{prop}[theorem]{Proposition}

\newtheorem{corollary}[theorem]{Corollary}
\newtheorem{conj}[theorem]{Conjecture}

\theoremstyle{definition}

\newtheorem{definition}{Definition}

\numberwithin{theorem}{section}
\numberwithin{definition}{section}
\numberwithin{remark}{section}

\begin{document}

\title[Perfectness of KR crystals]
{Perfectness of Kirillov--Reshetikhin crystals for nonexceptional types}

\author[G.~Fourier]{Ghislain Fourier}
\address{Mathematisches Institut der Universit\"at zu K\"oln,
Weyertal 86-90, 50931 K\"oln, Germany}
\email{gfourier@mi.uni-koeln.de}

\author[M.~Okado]{Masato Okado}
\address{Department of Mathematical Science,
Graduate School of Engineering Science, Osaka University,
Toyonaka, Osaka 560-8531, Japan}
\email{okado@sigmath.es.osaka-u.ac.jp}

\author[A.~Schilling]{Anne Schilling}
\address{Department of Mathematics, University of California, One Shields
Avenue, Davis, CA 95616-8633, U.S.A.}
\email{anne@math.ucdavis.edu}
\urladdr{http://www.math.ucdavis.edu/\~{}anne}

\thanks{\textit{Date:} November 2008}

\begin{abstract}
For nonexceptional types, we prove a conjecture of Hatayama et al. about the 
prefectness of Kirillov--Reshetikhin crystals.
\end{abstract}

\maketitle

\section{Introduction}
Kirillov--Reshetikhin (KR) crystals $B^{r,s}$ are finite affine crystals corresponding to finite-dimensional 
$U_q'(\geh)$-modules~\cite{CP,CP:1998}, where $\geh$ is an affine Kac--Moody algebra. 
Recently, a lot of progress has been made regarding these crystals which appear in mathematical 
physics and the path realization of affine highest weight crystals~\cite{KMN1:1992}. 
In~\cite{O:2007,OS:2008} it was shown that the KR crystals exist and in~\cite{FOS:2008} 
combinatorial realizations for these crystals were provided for all nonexceptional types. In this paper, 
we prove a conjecture of Hatayama et al.~\cite[Conjecture 2.1]{HKOTT:2002} about the perfectness 
of KR crystals.

\begin{conj} \cite[Conjecture 2.1]{HKOTT:2002} \label{conj:perfectness}
The Kirillov-Reshetikhin crystal $B^{r,s}$ is perfect if and only if $\frac{s}{c_r}$ is an integer
with $c_r$ as in Table~\ref{tab:c}. If $B^{r,s}$ is perfect, its level is $\frac{s}{c_r}$.
\end{conj}

In~\cite{KMN2:1992}, this conjecture was proven for type $A_n^{(1)}$, for $B^{1,s}$ for 
nonexceptional types (except for type $C_n^{(1)}$), for $B^{n-1,s}$, $B^{n,s}$ of type $D_n^{(1)}$, 
and $B^{n,s}$ for types $C_n^{(1)}$ and $D_{n+1}^{(2)}$. When the highest weight is given by 
the highest root, level-$1$ perfect crystals were constructed in~\cite{BFKL}. For $1\le r\le n-2$ for 
type $D_n^{(1)}$, $1\le r\le n-1$ for type $B_n^{(1)}$, and $1\le r\le n$ for type $A_{2n-1}^{(2)}$, 
the conjecture was proved in~\cite{S:2008}. The case $G_2^{(1)}$ and $r=1$ was treated 
in~\cite{Y:1998} and the case $D_4^{(3)}$ and $r=1$ was treated in~\cite{KMOY}. Naito and 
Sagaki~\cite{NS:2006} showed that the conjecture holds for twisted algebras, if it is true for the 
untwisted simply-laced cases.

In this paper we prove Conjecture~\ref{conj:perfectness} in general for nonexceptional types.

\begin{theorem} \label{thm:perfect}
If $\geh$ is of nonexceptional type, Conjecture~\ref{conj:perfectness} is true. 
\end{theorem}

The paper is organized as follows. In Section~\ref{sec:perfect} we give basic notation and the 
definition of perfectness in Definition~\ref{def:perfect}. In Section~\ref{sec:realization} we review
the realizations of the KR crystals of nonexceptional types as recently provided in~\cite{FOS:2008}.
Section~\ref{sec:proof} is reserved for the proof of Theorem~\ref{thm:perfect} and an explicit
description of the minimal elements $B_{\min}^{r,c_r s}$ of the perfect crystals. Several examples
for KR crystals of type $C_3^{(1)}$ are given in Section~\ref{sec:ex}.

\subsection*{Acknowledgements.}
GF was supported in part by DARPA and AFOSR through the grant FA9550-07-1-0543 and by
the DFG-Projekt ``Kombinatorische Beschreibung von Macdonald und Kostka-Foulkes Polynomen''.
MO was supported by grant JSPS 20540016.
AS was partially supported by the NSF grants DMS--0501101, DMS--0652641, and DMS--0652652.
We would like to thank the organizers of the conference``Quantum affine Lie algebras, extended affine 
Lie algebras, and applications'' held at Banff where part of this work was carried out and presented. 
Figure~\ref{fig:B21} was produced using MuPAD-Combinat/Sage-combinat.

\begin{table}
\begin{align*}
\begin{array}{|c|c|c}
\hline
& (c_1,\ldots,c_n)\\[1mm]
\hline \hline
A_n^{(1)} & (1,\ldots,1)\\[1mm] \hline
B_n^{(1)}   & (1,\ldots,1,2)\\[1mm] \hline
C_n^{(1)}   & (2,\ldots,2,1)\\[1mm] \hline
D_n^{(1)}   & (1,\ldots,1)\\[1mm] \hline
A_{2n-1}^{(2)} & (1,\ldots,1)\\[1mm] \hline
A_{2n}^{(2)}  & (1,\ldots,1)\\[1mm] \hline
D_{n+1}^{(2)}  & (1,\ldots,1)\\[1mm] \hline
\end{array}
\end{align*} 
\caption{List of $c_r$ \label{tab:c}}
\end{table}

\section{Definitions and perfectness}
\label{sec:perfect}
We follow the notation of~\cite{Kac,FOS:2008}. Let $\mathcal{B}$ be a 
$U_q'(\geh)$-crystal~\cite{Ka:1991}. Denote by $\alpha_i$ and $\La_i$ for $i\in I$ the 
simple roots and fundamental weights and by $c$ the canoncial central element associated to 
$\geh$, where $I$ is the index set of the Dynkin diagram of $\geh$ (see Table~\ref{tab:Dynkin}).
Let $P = \oplus_{i\in I} \Z \La_i$ be the weight lattice of $\geh$ and $P^+$ the set of dominant 
weights. For a positive integer $\ell$, the set of level-$\ell$ weights is
\begin{equation*}
	P^+_\ell = \{ \La \in P^+ \mid \lev(\La)=\ell \}.
\end{equation*}
where $\lev(\La):=\La(c)$. The set of level-$0$ weights is denoted by $P_0$.

We denote by $f_i, e_i: \mathcal{B} \to \mathcal{B} \cup \{\emptyset\}$ for $i\in I$ the Kashiwara 
operators and by $\wt:\mathcal{B}\to P$ the weight function on the crystal. For $b\in \mathcal{B}$ 
we define $\ve_i(b) = \max\{k \mid e_i^k(b) \neq \emptyset \}$, $\vp_i(b) = \max\{k \mid f_i^k(b) \neq 
\emptyset \}$, and
\begin{equation*}
 \ve(b) = \sum_{i\in I} \ve_i(b) \La_i \quad \text{and} \quad 
 \vp(b) = \sum_{i\in I} \vp_i(b) \La_i.
\end{equation*}

Next we define perfect crystals, see for example~\cite{HK:2002}.
\begin{definition} \label{def:perfect}
For a positive integer $\ell > 0$, a crystal $\mathcal{B}$ is called perfect crystal of level $\ell$, if the
following conditions are satisfied:
\begin{enumerate}
\item $\mathcal{B}$ is isomorphic to the crystal graph of a finite-dimensional 
$U_q^{'}(\mathfrak{g})$-module.
\item $\mathcal{B}\otimes \mathcal{B}$ is connected.
\item There exists a $\lambda\in P_0$, such that
$\wt(\mathcal{B}) \subset \lambda + \sum_{i\in I \setminus \{0\}} \mathbb{Z}_{\le 0} \alpha_i$ and there 
is a unique element in $\mathcal{B}$ of classical weight $\la$.
\item $\forall \; b \in \mathcal{B}, \;\; \lev(\varepsilon (b)) \geq \ell$.
\item $\forall \; \La \in P_\ell^{+}$, there exist unique elements
$b_{\La}, b^{\La} \in \mathcal{B}$, such that
$$ \ve ( b_{\La}) = \La = \vp( b^{\La}). $$
\end{enumerate}
\end{definition}

We denote by $\mathcal{B}_{\min}$ the set of minimal elements in $\mathcal{B}$, namely
\begin{equation*}
	\mathcal{B}_{\min} = \{ b\in \mathcal{B} \mid \lev(\ve(b))=\ell \}.
\end{equation*}
Note that condition (5) of Definition~\ref{def:perfect} ensures that $\ve,\vp:\mathcal{B}_{\min}
\to P_\ell^+$ are bijections. They induce an automorphism $\tau = \ve \circ \vp^{-1}$ on
$P_\ell^+$.

\begin{table}
\begin{eqnarray*}
A_n^{(1)}
&\vcenter{\xymatrix@R=1ex{
&&*{\circ}<3pt> \ar@{-}[drr]^<{\;\,0} \ar@{-}[dll] \\
*{\circ}<3pt> \ar@{-}[r]_<{1} &*{\circ}<3pt> \ar@{-}[r]_<{2}
&{} \ar@{.}[r]&{} \ar@{-}[r] &*{\circ}<3pt> \ar@{}[r]_<{n} &{}}}\\ 
B_n^{(1)}
&\vcenter{\xymatrix@R=1ex{
*{\circ}<3pt> \ar@{-}[dr]^<{0} \\
& *{\circ}<3pt> \ar@{-}[r]_<{2}
& {} \ar@{.}[r]&{}  \ar@{-}[r]_>{\,\,\,\,n-1} &
*{\circ}<3pt> \ar@{=}[r] |-{\scalebox{2}{\object@{>}}}& *{\circ}<3pt>\ar@{}_<{n} \\
*{\circ}<3pt> \ar@{-}[ur]_<{1}}}\\ 
C_n^{(1)}
&\vcenter{\xymatrix@R=1ex{
*{\circ}<3pt> \ar@{=}[r] |-{\scalebox{2}{\object@{>}}}_<{0}
&*{\circ}<3pt> \ar@{-}[r]_<{1}
& {} \ar@{.}[r]&{}  \ar@{-}[r]_>{\,\,\,\,n-1} &
*{\circ}<3pt> \ar@{=}[r] |-{\scalebox{2}{\object@{<}}}
& *{\bullet}<3pt>\ar@{}_<{n}}}\\ 
D_n^{(1)}
&\vcenter{\xymatrix@R=1ex{
*{\circ}<3pt> \ar@{-}[dr]^<{0}&&&&&*{\bullet}<3pt> \ar@{-}[dl]^<{n-1}\\
& *{\circ}<3pt> \ar@{-}[r]_<{2}
& {} \ar@{.}[r]&{} \ar@{-}[r]_>{\,\,\,n-2} &
*{\circ}<3pt> & \\
*{\circ}<3pt> \ar@{-}[ur]_<{1}&&&&&
*{\bullet}<3pt> \ar@{-}[ul]_<{n}}}\\ 
A_{2n}^{(2)}
&\vcenter{\xymatrix@R=1ex{
*{\circ}<3pt> \ar@{=}[r] |-{\scalebox{2}{\object@{<}}}_<{0}
&*{\circ}<3pt> \ar@{-}[r]_<{1}
& {} \ar@{.}[r]&{}  \ar@{-}[r]_>{\,\,\,\,n-1} &
*{\circ}<3pt> \ar@{=}[r] |-{\scalebox{2}{\object@{<}}}
& *{\circ}<3pt>\ar@{}_<{n}}}\\ 
A_{2n-1}^{(2)}
&\vcenter{\xymatrix@R=1ex{
*{\circ}<3pt> \ar@{-}[dr]^<{0} \\
& *{\circ}<3pt> \ar@{-}[r]_<{2}
& {} \ar@{.}[r]&{}  \ar@{-}[r]_>{\,\,\,\,n-1} &
*{\circ}<3pt> \ar@{=}[r] |-{\scalebox{2}{\object@{<}}}& *{\circ}<3pt>\ar@{}_<{n} \\
*{\circ}<3pt> \ar@{-}[ur]_<{1}}}\\ 
D_{n+1}^{(2)}
&\vcenter{\xymatrix@R=1ex{
*{\circ}<3pt> \ar@{=}[r] |-{\scalebox{2}{\object@{<}}}_<{0}
&*{\circ}<3pt> \ar@{-}[r]_<{1}
& {} \ar@{.}[r]&{}  \ar@{-}[r]_>{\,\,\,\,n-1} &
*{\circ}<3pt> \ar@{=}[r] |-{\scalebox{2}{\object@{>}}}
& *{\bullet}<3pt>\ar@{}_<{n}}}\\ 
\end{eqnarray*}
\caption{\label{tab:Dynkin}Dynkin diagrams}
\end{table}

In~\cite{S:2008,FOS:2008} $\pm$-diagrams were introduced, which describe the branching 
$X_n\to X_{n-1}$ where $X_n=B_n,C_n,D_n$.
A $\pm$-diagram $P$ of shape $\La/\la$ is a sequence of partitions $\la\subset \mu \subset \La$ 
such that $\La/\mu$ and $\mu/\la$ are horizontal strips (i.e. every column contains at most one box). 
We depict this $\pm$-diagram by the skew tableau of shape $\La/\la$ inwhich the cells of $\mu/\la$ are filled with the symbol $+$ and those of $\La/\mu$ are filled with the symbol $-$. There are further type specific rules which can be
found in~\cite[Section 3.2]{FOS:2008}. There exists a bijection $\bij$ between $\pm$-diagrams and
$X_{n-1}$-highest weight vectors.


\section{Realization of KR-crystals}
\label{sec:realization}
Throughout the paper we use the realization of $B^{r,s}$ as given 
in~\cite{FOS:2008, OS:2008, S:2008}. In this section we briefly recall the main constructions.

\subsection{KR crystals of type $A_{n}^{(1)}$}
Let $\Lambda=\ell_0\Lambda_0+\ell_1\Lambda_1+\cdots+\ell_n \Lambda_n$ be a dominant
weight. Then the level is given by 
\begin{equation*}
\begin{split}
    \lev(\La) &= \ell_0 + \cdots + \ell_n.
\end{split}
\end{equation*}

A combinatorial description of $B^{r,s}$ of type $A_n^{(1)}$ was 
provided by Shimozono~\cite{Sh:2002}. As a $\{1,2,\ldots, n\}$-crystal
\begin{equation*}
	B^{r,s} \cong B(s\La_r).
\end{equation*}
The Dynkin diagram of $A_n^{(1)}$ has a cyclic automorphism $\sigma (i) = i+1 \pmod{n+1}$ which extends to the crystal in form of the promotion operator. The action of the affine crystal operators $f_0$ 
and $e_0$ is given by
\begin{equation*}
	f_0 = \sigma^{-1} \circ f_1 \circ \sigma \qquad \text{and} \qquad e_0 = \sigma^{-1} \circ e_1 \circ \sigma.
\end{equation*}

\subsection{KR crystals of type $D_n^{(1)}$, $B_n^{(1)}$, $A_{2n-1}^{(2)}$}
Let $\Lambda=\ell_0\Lambda_0+\ell_1\Lambda_1+\cdots+\ell_n \Lambda_n$ be a dominant
weight. Then the level is given by
\begin{align*}
	\lev(\La) &= \ell_0+\ell_1+2\ell_2+2\ell_3+\cdots+2\ell_{n-2}+\ell_{n-1}+\ell_n && \text{for type $D_n^{(1)}$}\\
	\lev(\La) &= \ell_0+\ell_1+2\ell_2+2\ell_3+\cdots+2\ell_{n-2}+2\ell_{n-1}+\ell_n && \text{for type $B_n^{(1)}$}\\
	\lev(\La) &= \ell_0+\ell_1+2\ell_2+2\ell_3+\cdots+2\ell_{n-2}+2\ell_{n-1}+2\ell_n && \text{for type $A_{2n-1}^{(2)}$.}
\end{align*}
We have the following realization of $B^{r,s}$. Let $X_n=D_n,B_n,C_n$ be the classical subalgebra for
$D_n^{(1)}$, $B_n^{(1)}$, $A_{2n-1}^{(2)}$, respectively.
\begin{definition} \label{def:DBA}
Let $1\le r\le n-2$ for type $D_n^{(1)}$, $1\le r\le n-1$ for type $B_n^{(1)}$, and 
$1\le r\le n$ for type $A_{2n-1}^{(2)}$. Then $B^{r,s}$ is defined as follows. As an $X_n$-crystal
\begin{equation} \label{eq:V DBA decomp}
	B^{r,s} \cong \bigoplus_\La B(\La),
\end{equation}
where the sum runs over all dominant weights $\La$ that can be obtained from $s\La_r$ by the 
removal of vertical dominoes. The affine crystal operators $e_{0}$ and $f_{0}$ are defined as
\begin{equation} \label{eq:e0}
f_{0} = \sigma \circ f_{1} \circ \sigma \quad \text{and} \quad e_{0} = \sigma \circ e_{1} \circ \sigma,
\end{equation}
where $\sigma$ is the crystal automorphism defined in~\cite[Definition 4.2]{S:2008}.
\end{definition}

\begin{definition}
\label{def:Vn B}
Let $B_{A_{2n-1}^{(2)}}^{n,s}$ be the $A_{2n-1}^{(2)}$-KR crystal. Then $B^{n,s}$ of type $B_n^{(1)}$
is defined through the unique injective map $S:B^{n,s}\rightarrow B_{A_{2n-1}^{(2)}}^{n,s}$ such that 
\[
S(e_ib)=e_i^{m_i}S(b),\quad S(f_ib)=f_i^{m_i}S(b)\quad\text{for }i\in I,
\]
where $(m_i)_{0\le i\le n}=(2,2,\ldots,2,1)$. 
\end{definition}

In addition, the $\pm$-diagrams of $A_{2n-1}^{(2)}$ that occur in the image are precisely those 
which can be obtained by doubling a $\pm$-diagram of $B^{n,s}$ (see \cite[Lemma 3.5]{FOS:2008}). 
$S$ induces an embedding of dominant weights of $B_{n}^{(1)}$ into dominant weights of 
$A_{2n-1}^{(2)}$, namely $S(\La_i)=m_i\La_i$. It is easy to see that for any $\La\in P^+$ we have
$\lev(S(\Lambda)) = 2\, \lev(\Lambda)$.

For the definition of $B^{n,s}$ and $B^{n-1,s}$ of type $D_n^{(1)}$, see for 
example~\cite[Section 6.2]{FOS:2008}.

\subsection{KR crystal of type $C_n^{(1)}$}
The level of a dominant $C_n^{(1)}$ weight $\La = \ell_0 \La_0 + \cdots + \ell_n \La_n$
is given by
\begin{equation*}
\begin{split}
    \lev(\La) &= \ell_0 + \cdots + \ell_n.
\end{split}
\end{equation*}

We use the realization of $B^{r,s}$ as the fixed point set of the automorphism $\sigma$ 
\cite[Definition 4.2]{S:2008} (see Definition~\ref{def:DBA}) inside  
$B_{A_{2n+1}^{(2)}}^{r,s}$ of~\cite[Theorem 5.7]{FOS:2008}.
\begin{definition} \label{def:KR C}
For $1\le r<n$, the KR crystal $B^{r,s}$ of type $C_n^{(1)}$ is defined to be the fixed point set 
under $\sigma$ inside $B_{A_{2n+1}^{(2)}}^{r,s}$ with the operators
\begin{equation*}
    e_i = \begin{cases} e_0 e_1 & \text{for $i=0$,}\\
                            e_{i+1}   & \text{for $1\le i\le n$,}
              \end{cases}
\end{equation*}
where the Kashiwara operators on the right act in $B_{A_{2n+1}^{(2)}}^{r,s}$.
Under the crystal embedding $S: B^{r,s} \to B_{A_{2n+1}^{(2)}}^{r,s}$ we have
\begin{equation*}
    \La_i \mapsto \begin{cases} \La_0 + \La_1 & \text{for $i=0$,}\\
                                    \La_{i+1}     & \text{for $1\le i\le n$.}
                      \end{cases}
\end{equation*}
\end{definition}
Under the embedding $S$, the level of $\La\in P^+$ doubles, that is $\lev(S(\La))=2\, \lev(\La)$.


For $B^{n,s}$ of type $C_n^{(1)}$ we refer to~\cite[Section 6.1]{FOS:2008}.

\subsection{KR crystals of type $A_{2n}^{(2)}$, $D_{n+1}^{(2)}$}
Let $\Lambda=\ell_0\Lambda_0+\ell_1\Lambda_1+\cdots+\ell_n \Lambda_n$ be a dominant
weight. The level is given by
\begin{align*}
  	\lev(\La) &= \ell_0+2 \ell_1+2\ell_2+\cdots+2\ell_{n-2}+2\ell_{n-1}+2\ell_n && \text{for type } A_{2n}^{(2)}\\
		\lev(\La) &= \ell_0+2\ell_1+2\ell_2+\cdots+2\ell_{n-2}+2\ell_{n-1}+\ell_n && \text{for type }D_{n+1}^{(2)}.
\end{align*}	
Define positive integers $m_i$ for $i\in I$ as follows:
\begin{equation}
\label{eq:m}
(m_0,m_1,\ldots,m_{n-1},m_n) = \begin{cases}
      (1,2,\ldots,2,2) & \text{ for }A_{2n}^{(2)},\\
      (1,2,\ldots,2,1) & \text{ for }D_{n+1}^{(2)}.
\end{cases}
\end{equation}
Then $B^{r,s}$ can be realized as follows.
\begin{definition}
\label{def:embedding}
For $1\le r\le n$ for $\geh=A_{2n}^{(2)}$, $1\le r<n$ for $\geh=D_{n+1}^{(2)}$ and $s\ge1$,
there exists a unique injective map
$S:B^{r,s}_\geh\longrightarrow B_{C_n^{(1)}}^{r,2s}$ such that
\[
 S(e_ib)=e_i^{m_i}S(b),\quad S(f_ib)=f_i^{m_i}S(b) \qquad \text{for $i \in I$.}
\]
\end{definition}
The $\pm$-diagrams of $C_{n}^{(1)}$ that occur in the image of $S$ are precisely those which can be obtained by doubling a $\pm$-diagram of $B^{r,s}$ (see \cite[Lemma 3.5]{FOS:2008}).
$S$ induces an embedding of dominant weights for $A_{2n}^{(2)},D_{n+1}^{(2)}$ into dominant weights of type $C_{n}^{(1)}$, with $S(\Lambda_i) = m_i \Lambda_i$. This map preserves the level of a weight,
that is $\lev(S(\La))=\lev(\La)$.

For the case $r = n$ of type $D_{n+1}^{(2)}$ we refer to \cite[Definition 6.2]{FOS:2008}.

\section{Proof of Theorem~\ref{thm:perfect}}
\label{sec:proof}
For type $A_n^{(1)}$, perfectness of $B^{r,s}$ was proven in~\cite{KMN2:1992}. For all other types,
in the case that $\frac{s}{c_r}$ is an integer, we need to show that the 5 defining conditions in 
Definition~\ref{def:perfect} are satisfied:
\begin{enumerate}
\item This was recently shown in \cite{OS:2008}.
\item This follows from~\cite[Corollary 6.1]{FSS:2007} under~\cite[Assumption 1]{FSS:2007}.
Assumption 1 is satisfied except for type $A_{2n}^{(2)}$: The regularity of $B^{r,s}$ is ensured by (1),
the existence of an automorphism $\sigma$ was proven in~\cite[Section 7]{FOS:2008}, and
the unique element $u\in B^{r,s}$ such that $\ve(u)=s\La_0$ and $\vp(u)=s\La_\nu$ 
(where  $\nu=1$ for $r$ odd for types $B_n^{(1)}$, $D_n^{(1)}$, $A_{2n-1}^{(2)}$,
$\nu=r$ for $A_n^{(1)}$, and $\nu=0$ otherwise) is given by the classically 
highest weight element in the component $B(0)$ for $\nu=0$, $B(s\La_1)$ for 
$\nu=1$, and $B(s\La_r)$ for $\nu=r$. Note that $\La_0=\tau(\La_\nu)$, where $\tau=\ve\circ \vp^{-1}$.
For type $A_{2n}^{(2)}$, perfectness follows from~\cite{NS:2006}.
\item The statement is true for $\lambda = s (\Lambda_r-\Lambda_r(c) \La_0)$, which follows from 
the decomposition formulas~\cite{Chari:2001,H:2006,H:2007,Nakajima:2003}. 
\end{enumerate}
Conditions (4) and (5) will be shown in the following subsections using case by case
considerations: Section~\ref{subsec:A} for type $A_n^{(1)}$, Sections~\ref{subsec:BDA},
\ref{subsec:Dn}, and~\ref{subsec:Bn} for types $B_n^{(1)}$, $D_n^{(1)}$, $A_{2n-1}^{(2)}$,
Sections~\ref{subsec:C} and~\ref{subsec:Cn} for type $C_n^{(1)}$, Section~\ref{subsec:A2}
for type $A_{2n}^{(2)}$, and Sections~\ref{subsec:D2} and~\ref{subsec:D2n} for type
$D_{n+1}^{(2)}$.

When $\frac{s}{c_r}$ is not an integer, we show in the subsequent sections
that the minimum of the level of $\ve(b)$ is the smallest integer exceeding $\frac{s}{c_r}$, and 
provide examples that contradict condition (5) of Definition~\ref{def:perfect} for each crystal,
thereby proving that $B^{r,s}$ is not perfect. In the case that $\frac{s}{c_r}$ is an integer,
we provide an explicit construction of the minimal elements of $B^{r,s}$.

\subsection{Type $A_n^{(1)}$}
\label{subsec:A}
It was already proven in~\cite{KMN2:1992} that $B^{r,s}$ is perfect.
We give below its associated automorphism $\tau$ and minimal elements. 
$\tau$ on $P$ is defined by
\[
\tau(\sum_{i=0}^nk_i\La_i)=\sum_{i=0}^nk_i\La_{i-r\,\mathrm{mod}\,n+1}.
\]
Recall that $B^{r,s}$ is identified with the set of
semistandard tableaux of $r\times s$ rectangular shape over the alphabet
$\{1,2,\ldots,n+1\}$. For $b\in B^{r,s}$ let $x_{ij}=x_{ij}(b)$ denote the number of
letters $j$ in the $i$-th row of $b$ for $1\le i\le r,1\le j\le n+1$. Set $r'=n+1-r$, then 
\[
x_{ij}=0\quad\text{unless}\quad i\le j\le i+r'.
\]
Let $\La=\sum_{i=0}^n\ell_i\La_i$ be
in $P^+_s$, that is, $\ell_0,\ell_1,\ldots,\ell_n\in\Z_{\ge0},\sum_{i=0}^n\ell_i=s$.
Then $x_{ij}(b)$ of the minimal element $b$ such that $\ve(b)=\La$ is
given by
\begin{equation}
\label{eq:min elem A}
\begin{split}
x_{ii}&=\ell_0+\sum_{\alpha=i}^{r-1}\ell_{\alpha+r'},\\
x_{ij}&=\ell_{j-i}\quad(i<j<i+r'),\\
x_{i,i+r'}&=\sum_{\alpha=0}^{i-1}\ell_{\alpha+r'}
\end{split}
\end{equation}
for $1\le i\le r$.

\subsection{Types $B_n^{(1)}$, $D_n^{(1)}$, $A_{2n-1}^{(2)}$}
\label{subsec:BDA}
Conditions (4) and (5) of Definition~\ref{def:perfect} for $1\le r\le n-2$ for type $D_n^{(1)}$,
$1\le r\le n-1$ for type $B_n^{(1)}$, and $1\le r\le n$ for type $A_{2n-1}^{(2)}$ were shown
in~\cite[Section 6]{S:2008}. We briefly review the construction of the minimal elements here since
they are important in the construction of the minimal elements for type $C_n^{(1)}$.

To a given fundamental weight $\Lambda_k$ we may associate the following $\pm$-diagram
\begin{equation} \label{eq:La pm}
\diagram: \Lambda_k \mapsto \begin{cases}
	\phantom{k+1\left\{\right.} \emptyset & \text{if $r$ is even and $k=0$}\\[1mm]
	\phantom{k+1\left\{\right.} \young(-,+) & \text{if $r$ is even and $k=1$}\\[3mm]
	\phantom{k+1\left\{\right.} \young(+) & \text{if $r$ is odd and $k=0$}\\[2mm]
	\phantom{k+1\left\{\right.} \young(-) & \text{if $r$ is odd and $k=1$}\\[2mm]
	k+1\left\{\young(-,+,\mb\mb,\mb\mb)\right. 
		& \text{if $k\not \equiv r \bmod{2}$ and $2\le k\le r$}\\[8mm]
	\qquad k\left\{\young(+-,\mb\mb,\mb\mb,\mb\mb)\right.
		& \text{if $k \equiv r \bmod{2}$ and $2\le k\le r$}\\[8mm]
	\qquad r\left\{\young(\mb\mb,\mb\mb,\mb\mb,\mb\mb)\right.
		& \begin{array}{l} \text{if $r<k\le n-2$ for type $D_n^{(1)}$}\\
		                                \text{if $r<k\le n-1$ for type $B_n^{(1)}$}\\
		                                \text{if $r<k\le n$ for type $A_{2n-1}^{(2)}$} \end{array} \\[8mm]
	\qquad r\left\{\young(\mb,\mb,\mb,\mb)\right. 
		& \begin{array}{l} \text{if $k=n-1,n$ for type $D_n^{(1)}$}\\
		                                \text{if $k=n$ for type $B_n^{(1)}$.} \end{array}
\end{cases}
\end{equation}
This map can be extended to any dominant weight $\Lambda=\ell_0\Lambda_0+\cdots+
\ell_n\Lambda_n$ by concatenating the columns of the $\pm$-diagrams of each piece. 

To every fundamental weight $\Lambda_k$ we also associate a string of operators $f_i$ with 
$i\in\{2,3,\ldots,n\}$ as follows. Let $T(\Lambda_k)$ be the tableau assigned to $\Lambda_k$ as
\begin{equation*}
T(\Lambda_k) = \begin{cases}
	\quad u & \text{if $r$ is even and $k=0$}\\[1mm]
	\quad \young(\bb,2) & \text{if $r$ is even and $k=1$}\\[4mm]
	\quad \young(1) & \text{if $r$ is odd and $k=0$}\\[1mm]
	\quad \young(\ab) & \text{if $r$ is odd and $k=1$}\\[2mm]
	\quad \begin{array}{|c|c|} \cline{1-1} \overline{k+1} & \multicolumn{1}{c}{} \\ \cline{1-1} 
                                        k+1 & \multicolumn{1}{c}{}\\ \hline k&\bar{2}\\
                                        \hline \vdots & \vdots\\ \hline 2& \bar{k}\\
                                         \hline \end{array} & \text{if $2\le k\le r$ and $k\not \equiv r\bmod 2$} \\[1.4cm]
          \quad \begin{array}{|c|c|} \hline k & \ab\\ \hline \vdots & \vdots\\[2mm] \hline 1& \bar{k}\\
                                         \hline \end{array} & \text{if $2\le k\le r$ and $k\equiv r\bmod 2$} \\[1cm]
	\quad \begin{array}{|c|c|} \hline k&\overline{k-r+1}\\ \hline \vdots&\vdots\\ \hline
                                        k-r+1 & \overline{k} \\ \hline \end{array} 
                     & \begin{array}{l} \text{if $r<k\le n-2$ for type $D_n^{(1)}$}\\
                     			      \text{if $r<k\le n-1$ for type $B_n^{(1)}$}\\
			      		      \text{if $r<k\le n$ for type $A_{2n-1}^{(2)}$} \end{array} \\[0.9cm]
          \quad \left.\begin{array}{|c|} \hline \vdots \\ \hline n\\ \hline 
                                         \overline{n}\\ \hline n \\ \hline\end{array} \right\} r 
                     & \text{for $k=n-1$ for type $D_n^{(1)}$}\\[1.1cm]
          \quad \text{previous case with $n\leftrightarrow \bar{n}$}
                & \text{for $k=n$ for type $D_n^{(1)}$}\\[.2cm]
           \quad \left.\begin{array}{|c|} \hline 0 \\ \hline \vdots \\ \hline 0 \\ \hline \end{array} \right\} r 
                     & \text{for $k=n$ for type $B_n^{(1)}$}
\end{cases}
\end{equation*}
Then $f(\Lambda_k)$ for $0\le k\le n$ is defined such that
$T(\Lambda_k) = f(\Lambda_k) \bij(\diagram(\Lambda_k))$, where $\bij$ is the bijection between $\pm$-diagrams and $X_{n-1}$-highest weight elements 
(see~\cite{S:2008,FOS:2008}). Note that in fact $f(\Lambda_0)=f(\Lambda_1)=1$.

The minimal element $b$ in $B^{r,s}$ that satisfies $\varepsilon(b)=\Lambda$ can now be
constructed as follows
\begin{equation*}
	b = f(\Lambda_n)^{\ell_n} \cdots  f(\Lambda_2)^{\ell_2} \bij(\diagram(\Lambda)).
\end{equation*}

{}From the condition that $\wt(b)=\vp(b)-\ve(b)$ it is not hard to see that $\vp(b)=\ve(b)$
for $b\in B_{\min}^{r,s}$ and $r$ even. For $r$ odd, we have $\vp(b)=\sigma \circ \sigma' \circ \ve(b)$
for $b\in B_{\min}^{r,s}$, where $\sigma$ is the Dynkin diagram automorphism interchanging nodes
$0$ and $1$, $\sigma'$ is the Dynkin diagram automorphism interchanging nodes $n-1$ and $n$ 
for type $D_n^{(1)}$, and $\sigma'$ is the identity for type $B_n^{(1)}$ and $A_{2n-1}^{(2)}$. Hence, 
for $\La=\sum_{i=0}^n\ell_i\La_i\in P_s^+$, we have
\begin{equation*}
\tau(\La)=\begin{cases}
	\La & \text{if $r$ is even},\\
 	\ell_0\La_1+\ell_1\La_0+\sum_{i=2}^n\ell_i\La_i
	 & \text{if $r$ is odd,}\\
	 & \text{\quad types $B_n^{(1)}, A_{2n-1}^{(2)}$},\\
	\ell_0\La_1+\ell_1\La_0+\sum_{i=2}^{n-2}\ell_i\La_i+\ell_{n-1}\La_n+\ell_n\La_{n-1}
	 &\text{if $r$ is odd, type $D_n^{(1)}$}.
\end{cases}
\end{equation*}

\subsection{Type $D_n^{(1)}$ for $r=n-1,n$}
\label{subsec:Dn}
The cases when $r=n,n-1$ for type $D_n^{(1)}$ were treated in~\cite{KMN2:1992}. We will give 
the minimal elements below.
Since $B^{n,s}$ and $B^{n-1,s}$ are related via the Dynkin diagram automorphism interchanging 
$\La_n$ and $\La_{n-1}$, we only deal with $B^{n,s}$. 
As a $D_n$-crystal it is isomorphic to $B(s\La_n)$. There is a description of an element in terms
of semistandard tableau of $n\times s$ rectangular shape with letters from the alphabet $\mathcal{A}
=\{1,2,\ldots,n,\ol{n},\ldots,\ol{1}\}$ with partial order
\[
1<2<\cdots<n-1<{n\atop\ol{n}}<\ol{n-1}<\cdots<\ol{1}.
\]
Moreover, each column does not contain both $k$ and $\ol{k}$.
Let $c_i$ be the $i$th column. Then the number of barred letters in $c_i$ is even, and
the action of $e_i,f_i$ ($i=1,\ldots,n$) is calculated through that of
$c_s\otimes\cdots\otimes c_1$ of $B(\La_n)^{\otimes s}$. With this realization the minimal element 
$b_{\La}$ such that $\ve(b_{\La})=\La=\sum_{i=0}^n\ell_i\La_i$ ($\ell_i\in\Z_{\ge0},
\lev(\La)=s$) is given as follows. 
Let $x_{ij}$ ($1\le i\le n,j\in\mathcal{A}$) be the number of $j$ 
in the $i$th row. $x_{ij}=0$ unless $i\le j\le\ol{n-i+1}$. The other $x_{ij}$ of $b_{\La}$
is given by 
\begin{align*}
&x_{11}=\ell_0+\ell_2+\ell_3+\cdots+\ell_{n-2}+
\begin{cases} \ell_{n-1} & \text{for $n$ even,}\\
                          \ell_n & \text{for $n$ odd,} \end{cases}\\                          
&x_{1j}=\ell_{j-1}\;(2\le j\le n-1),\quad
(x_{1n},x_{1\ol{n}})=\begin{cases} (0,\ell_n) & \text{for $n$ even,}\\
                                                              (\ell_{n-1},0) & \text{for $n$ odd,} \end{cases}
\end{align*}
if $2\le i\le n-1$,
\begin{align*}
&x_{ii}=\ell_0+\ell_2+\ell_3+\cdots+\ell_{n-i},
  \quad x_{ij}=\ell_{j-i}\;(i+1\le j\le n-1),\\
&(x_{in},x_{i\ol{n}})=\begin{cases} (\ell_{n-i}+\ell_{n-i+1},0) & \text{$n-i$ even,}\\ 
                                                             (0,\ell_{n-i}+\ell_{n-i+1}) & \text{$n-i$ odd,}
\end{cases}\\
&x_{i\,\ol{j}}=\ell_{2n+1-i-j}\;(n-i+3\le j\le n-1),\quad
x_{i\,\ol{n-i+2}}=\begin{cases} \ell_{n-1} & \text{$n$ even,}\\  \ell_n & \text{$n$ odd,}
\end{cases}\\
&x_{i\,\ol{n-i+1}}=\ell_{n-i+1}+ \ell_{n-i+2} +\cdots+\ell_{n-2}+
\begin{cases} \ell_n & \text{$n$ even,}\\ \ell_{n-1} & \text{$n$ odd,}
\end{cases}
\end{align*}
and
\begin{align*}
&x_{nn}=\ell_0,\quad x_{n\ol{n}}=0,\quad x_{n\ol{j}}=\ell_{n+1-j}\;(3\le j\le n-1),\\
&x_{n\ol{2}}=\begin{cases} \ell_{n-1} & \text{$n$ even,}\\ \ell_n & \text{$n$ odd,}
\end{cases}
\quad
x_{n\ol{1}}=\ell_1+ \ell_2 +\cdots+\ell_{n-2}+\begin{cases} 
\ell_n & \text{$n$ even,}\\ \ell_{n-1} & \text{$n$ odd.} \end{cases}
\end{align*}
The automorphism $\tau$ is given by
\[
\tau\bigl(\sum_{i=0}^n\ell_i\La_i\bigr)=\ell_0\La_{n-1}+\ell_1\La_n+
\sum_{i=2}^{n-2}\ell_i\La_{n-i}+\begin{cases}
\ell_{n-1}\La_0+\ell_n\La_1 & \text{$n$ even,}\\
\ell_{n-1}\La_1+\ell_n\La_0 & \text{$n$ odd.} \end{cases}
\]

\subsection{Type $B_n^{(1)}$ for $r=n$}
\label{subsec:Bn}
In this section we consider the perfectness of $B^{n,s}$ of type $B_n^{(1)}$.
\begin{prop} \label{prop:bounds B}
We have
\begin{equation*}
\begin{split}
	&\min \{ \lev(\ve(b)) \mid b \in B^{n,2s+1} \} \geq s+1,\\
	&\min \{ \lev(\ve(b)) \mid b \in B^{n,2s} \} \geq s.
\end{split}
\end{equation*}
\end{prop} 

\begin{proof}
Suppose, there exists an element $b \in B^{n,2s+1}$ with $\lev(\varepsilon(b)) = p < s+1$. Since 
$B^{n,2s+1}$ is embedded into $B_{A_{2n-1}^{(2)}}^{n,2s+1}$ by Definition~\ref{def:Vn B}, this 
would yield an element $\tilde{b}\in B_{A_{2n-1}^{(2)}}^{n,2s+1}$ with $\lev(\tilde{b}) < 2s+1$. But this is 
not possible, since $B_{A_{2n-1}^{(2)}}^{n,2s+1}$ is a perfect crystal of level $2s +1$.

Suppose there exists an element $b \in B^{n,2s}$ with $\lev(\varepsilon(b)) = p < s$. By the same argument one obtains a contradiction to the level of $B_{A_{2n-1}^{(2)}}^{n,2s}$.
\end{proof}

Hence to show that $B^{n,2s+1}$ is not perfect, it is enough to provide two elements $b_1, b_2 \in 
B_{A_{2n-1}^{(2)}}^{n,2s+1}$  which are in the realization of $B^{r,s}$ under $S$ and satisfy 
$\ve(b_1) = \ve(b_2) = \Lambda$, where $\lev(\Lambda) = 2s+2$.

\begin{prop} \label{prop:B nonperfect}
Define the following elements $b_1, b_2 \in B_{A_{2n-1}^{(2)}}^{n,2s+1}$: For $n$ odd,
let $P_1$ be the $\pm$-diagram corresponding to one column of height $n$ with a $+$, and 
$2s$ columns of height $1$ with $-$ signs, and $P_2$ the analogous $\pm$-diagram but with 
a $-$ in the column of height $n$. Set 
$\vec{a} = (n (n-1)^2 n (n-2)^2 (n-1)^2 n \ldots 2^2\ldots (n-1)^2 n)$ and
\begin{equation*}
	b_1 = f_{\vec{a}}(\bij(P_1)) \quad \text{and} \quad b_2 = f_{\vec{a}}(\bij(P_2)).
\end{equation*} 
For $n$ even, replace the columns of height $1$ with columns of height $2$ and fill them with 
$\pm$-pairs.
Then $b_1, b_2 \in S(B^{n,2s+1})$ and $\ve(b_1) = \ve(b_2) = 2s \Lambda_1 + \Lambda_n$, which 
is of level $2s+2$.
\end{prop}

\begin{proof}
It is clear from the construction that the $\pm$-diagrams corresponding to $b_1$ and $b_2$ can 
be obtained by doubling a $B_{n}^{(1)}$ $\pm$-diagram (see~\cite[Lemma 3.5]{FOS:2008}). 
Hence $\bij(P_1), \bij(P_2) \in S(B^{n,2s+1})$. The sequence $\vec{a}$ can be obtained by doubling
a type $B_n^{(1)}$ sequence using $(m_1,m_2,\ldots,m_n)= (2,\ldots,2,1)$, so by 
Definition~\ref{def:Vn B} $b_1$ and $b_2$ are in the image of the embedding $S$ that realizes 
$B^{n, 2s+1}$. The claim that $\ve(b_1) = \ve(b_2) = 2s \Lambda_1 + \Lambda_n$ can be checked 
explicitly.
\end{proof}

\begin{corollary}
The KR crystal $B^{n,2s+1}$ of type $B_n^{(1)}$ is not perfect.
\end{corollary}
\begin{proof}
This follows directly from Proposition~\ref{prop:B nonperfect} using the embedding $S$ of
Definition~\ref{def:Vn B}.
\end{proof}

\begin{prop}\label{prop:minimal B Vn}
There exists a bijection, induced by $\varepsilon$, from $B_{\min}^{n,2s}$ to $P_s^+$.
Hence $B^{n,2s}$ is perfect of level $s$.
\end{prop}

\begin{proof}
Let $S$ be the embedding from Definition~\ref{def:Vn B}. Then we have an induced embedding of dominant weights $\Lambda$ of $B_n^{(1)}$ into dominant weights of $A_{2n-1}^{(2)}$ via the map 
$S$, that sends $\Lambda_i \mapsto m_i \Lambda_i$.

In~\cite[Section 6]{S:2008} (see Section~\ref{subsec:BDA}) the minimal elements for 
$A_{2n-1}^{(2)}$ were constructed by giving a $\pm$-diagram and a sequence from the 
$\{2,\ldots, n\}$-highest weight to the minimal element. Since $(m_0,\ldots,m_{n-1},m_n)=
(2,\ldots,2,1)$ and columns of height $n$ in~\eqref{eq:La pm} for type $A_{2n-1}^{(2)}$ are doubled, 
it is clear from the construction that the $\pm$-diagrams corresponding to weights $S(\La)$
are in the image of $S$ of $\pm$-diagrams for $B_n^{(1)}$ (see~\cite[Lemma 3.5]{FOS:2008}).
Also, since under $S$ all weights $\La_i$ for $1\le i<n$ are doubled, it follows that the
sequences are ``doubled" using the $m_i$.
Hence a minimal element of $B^{n,2s}$ of level $s$ is in one-to-one correspondence with 
those minimal elements in $B_{A_{2n-1}^{(2)}}^{n, 2s}$ that can be obtained from doubling
a $\pm$-diagram of $B^{n,2s}$.
This implies that $\varepsilon$ defines a bijection between $B_{\min}^{n,2s}$ and $P_s^+$.
\end{proof}

The automorphism $\tau$ of the perfect KR crystal $B^{n,2s}$ is given by 
\[
	\tau\bigl(\sum_{i=0}^n\ell_i\La_i\bigr)=\begin{cases}
		\sum_{i=0}^n\ell_i\La_i & \text{if $n$ is even},\\
		\ell_0\La_1+\ell_1\La_0+\sum_{i=2}^n\ell_i\La_i & \text{if $n$ is odd}.
	\end{cases}
\]

\subsection{Type $C_n^{(1)}$} 
\label{subsec:C}
In this section we consider $B^{r,s}$ of type $C_n^{(1)}$ for $r<n$.
\begin{prop} \label{prop:C bounds}
Let $r < n$. Then
\begin{equation*}
\begin{split}
	&\min \{ \lev(\ve(b)) \mid b \in B^{r,2s+1}\} \geq s+1,\\
	&\min \{ \lev(\ve(b)) \mid b \in B^{r,2s} \} \geq s.
\end{split}
\end{equation*}
\end{prop}

\begin{proof}
By Definition~\ref{def:KR C}, the crystal $B^{r,s}$ is realized inside $B_{A_{2n+1}^{(2)}}^{r,s}$.
The proof is similar to the proof of Proposition~\ref{prop:bounds B} for type $B_n^{(1)}$.
\end{proof}

Hence to show that $B^{r,2s+1}$ is not perfect, it is suffices to give two elements $b_1, b_2\in 
B_{A_{2n+1}^{(2)}}^{r,2s+1}$ that are fixed points under $\sigma$ with 
$\ve(b_1) = \ve(b_2) = \Lambda$, where $\lev(\Lambda) = 2s+2$.
\begin{prop}  \label{prop:C nonperfect}
Let $b_1, b_2 \in B_{A_{2n+1}^{(2)}}^{r,2s+1}$, where $b_1$ consists of $s$ columns of the form 
read from bottom to top $(1,2,\ldots,r)$, $s$ columns of the form $(\ol{r}, \ol{r-1},\ldots, \ol{1})$, 
and a column $(\ol{r+1},\ldots,\ol{2})$.
In $b_2$ the last column is replaced by $(r+2,\ldots,2r+2)$ if $2r+2\le n$ and
$(r+2,\ldots,n,\ol{n},\ldots,\ol{k})$ of height $n$ otherwise. Then
\begin{equation*}
	\ve(b_1) = \ve(b_2) = \begin{cases}
		s \La_r + \La_{r+1} & \text{if $r>1$,}\\
		s (\La_0 + \La_1) + \La_2 & \text{if $r=1$,}
	\end{cases}
\end{equation*}
which is of level $2s + 2$.
\end{prop}

\begin{proof}
The claim is easy to check explicitly.
\end{proof}

\begin{corollary}
The KR crystal $B^{n,2s+1}$ of type $C_n^{(1)}$ is not perfect.
\end{corollary}
\begin{proof}
The $\{2, \ldots, n\}$-highest weight elements in the same component as $b_1$ and $b_2$ 
of Proposition~\ref{prop:C nonperfect} correspond to $\pm$-diagrams that are invariant under 
$\sigma$. Hence, by Definition~\ref{def:KR C}, $b_1$ and $b_2$ are fixed points under $\sigma$.
Combining this result with Proposition~\ref{prop:C bounds} proves that $B^{r,2s+1}$ is not perfect.
\end{proof}

\begin{prop} \label{prop:minimal C}
There exists a bijection, induced by $\varepsilon$, from $B_{\min}^{r,2s}$ to $P_s^{+}$.
Hence $B^{r,2s}$ is perfect of level $s$.
\end{prop}

\begin{proof}
By Definition~\ref{def:KR C}, $B^{r,s}$ of type $C_n^{(1)}$ is realized inside
$B^{r,s}_{A_{2n+1}^{(2)}}$ as the fixed points under $\sigma$. Under the embedding $S$,
it is clear that a dominant weight $\La=\ell_0\La_0 + \ell_1 \La_1+\cdots+ \ell_{n+1}\La_{n+1}$
of type $A_{2n+1}^{(2)}$ is in the image if and only if $\ell_0=\ell_1$. Hence it is clear from the
construction of the minimal elements for $A_{2n+1}^{(2)}$ as described in Section~\ref{subsec:BDA} 
that the minimal elements corresponding to $\La$ with $\ell_0=\ell_1$ are invariant under $\sigma$.
By \cite[Theorem 6.1]{S:2008} there is a bijection between all dominant weights $\La$ 
of type $A_{2n+1}^{(2)}$ with $\ell_0=\ell_1$ and $\lev(\La)=2s$ and minimal elements in
$B_{A_{2n+1}^{(2)}}^{r,2s}$ that are invariant under $\sigma$. Hence using $S$, there is a bijection
between dominant weights in $P_s^+$ of type $C_n^{(1)}$ and 
$B_{\min}^{r,2s}$. 
\end{proof}

The automorphism $\tau$ of the perfect KR crystal $B^{r,2s}$ is given by the identity.

\subsection{Type $C_n^{(1)}$ for $r=n$}
\label{subsec:Cn}
This case is treated in~\cite{KMN2:1992}. For the minimal elements, we follow the construction in Section~\ref{subsec:BDA}. To every fundamental weight $\La_k$ we
associate a column tableau $T(\La_k)$ of height $n$ whose entries are $k+1,k+2,\ldots,n,
\ol{n},\ldots,\ol{n-k+1}$ ($1,2,\ldots,n$ for $k=0$) reading from bottom to top. Let $f(\La_k)$ be defined
such that $T(\La_k)=f(\La_k)b_1$, where $b_k$ is the highest weight tableau in $B(k\La_n)$. Then the
minimal element $b$ in $B^{n,s}$ such that $\ve(b)=\La=\sum_{i=0}^n\ell_i\La_i\in P_s^+$ 
is constructed as
\[
b=f(\La_n)^{\ell_n}\cdots f(\La_1)^{\ell_1}b_s.
\]
The automorphism $\tau$ is given by 
\[
\tau(\sum_{i=0}^n\ell_i\La_i)=\sum_{i=0}^n\ell_i\La_{n-i}.
\]

\subsection{Type $A_{2n}^{(2)}$}
\label{subsec:A2}
For type $A_{2n}^{(2)}$ one may use the result of Naito and Sagaki~\cite[Theorem 2.4.1]{NS:2006} 
which states that under their ~\cite[Assumption 2.3.1]{NS:2006}
(which requires that $B^{r,s}$ for $A_{2n}^{(1)}$ is perfect) all $B^{r,s}$ for
$A_{2n}^{(2)}$ are perfect. Here we provide a description of the minimal elements via the emebdding 
$S$ into $B_{C_n^{(1)}}^{r,2s}$.

\begin{prop}\label{prop:minimal A2n}
The minimal elements of $B^{r,s}$ of level $s$ are precisely those that corresponding to 
doubled $\pm$-diagrams in $B_{C_n^{(1)}}^{r,2s}$.
\end{prop}

\begin{proof}
In Proposition~\ref{prop:minimal C} a description of the minimal elements of $B_{C_n^{(1)}}^{r, 2s}$ 
is given. We have the realization of $B^{r,s}$ via the map $S$ from Definition~\ref{def:embedding}. 
In the same way as in the proof of Proposition~\ref{prop:minimal B Vn} one can show, that the minimal
elements of $B_{C_n^{(1)}}^{r, 2s}$ that correspond to doubled dominant weights are precisely those
in the realization of $B^{r,s}$, hence $\varepsilon$ defines a bijection between $B_{\min}^{r,s}$ and
$P_s^+$.
\end{proof}
The automorphism $\tau$ is given by the identity.

\subsection{Type $D_{n+1}^{(2)}$ for $r<n$}
\label{subsec:D2}
\begin{prop}\label{prop:minimal Dn+1} 
Let $r<n$. There exists a bijection $ B_{\min}^{r,s}$ to $P_s^{+}$, defined by $\varepsilon$.
Hence $B^{r,s}$ is perfect.
\end{prop}

\begin{proof}
This proof is analogous to the proof of Proposition~\ref{prop:minimal A2n}.
\end{proof}
The automorphism $\tau$ is given by the identity.

\subsection{Type $D_{n+1}^{(2)}$ for $r=n$}
\label{subsec:D2n}
This case is already treated in~\cite{KMN2:1992}, which we summarize below.
As a $B_n$-crystal it is isomorphic to $B(s\La_n)$. There is a description of its elements in terms
of semistandard tableaux of $n\times s$ rectangular shape with letters from the alphabet $\mathcal{A}
=\{1<2<\cdots<n<\ol{n}<\cdots<\ol{1}\}$. Moreover, each column does not contain both $k$ and $\ol{k}$.
Let $c_i$ be the $i$th column, then the action of $e_i,f_i$ ($i=1,\ldots,n$) is calculated through that of
$c_s\otimes\cdots\otimes c_1$ of $B(\La_n)^{\otimes s}$. With this realization the minimal element 
$b_{\La}$ such that $\ve(b_{\La})=\La =\sum_{i=0}^n \ell_i \La_i \in P_s^+$ is given as follows. 
Let $x_{ij}$ ($1\le i\le n,j\in\mathcal{A}$) be the number of $j$ 
in the $i$th row. Note that $x_{ij}=0$ unless $i\le j\le\ol{n-i+1}$. The table $(x_{ij})$ of $b_{\La}$
is then given by $x_{ii}=\ell_0+\cdots+\ell_{n-i}\,(1\le i\le n)$, $x_{ij}=\ell_{j-i}\,(i+1\le j\le n)$,
$x_{i\ol{j}}=\ell_j+\cdots+\ell_n\,(n-i+1\le j\le n)$. The automorphism $\tau$ is given by
\[
\tau(\sum_{i=0}^n\ell_i\La_i)=\sum_{i=0}^n\ell_i\La_{n-i}.
\]

\section{Examples for type $C_3^{(1)}$}
\label{sec:ex}
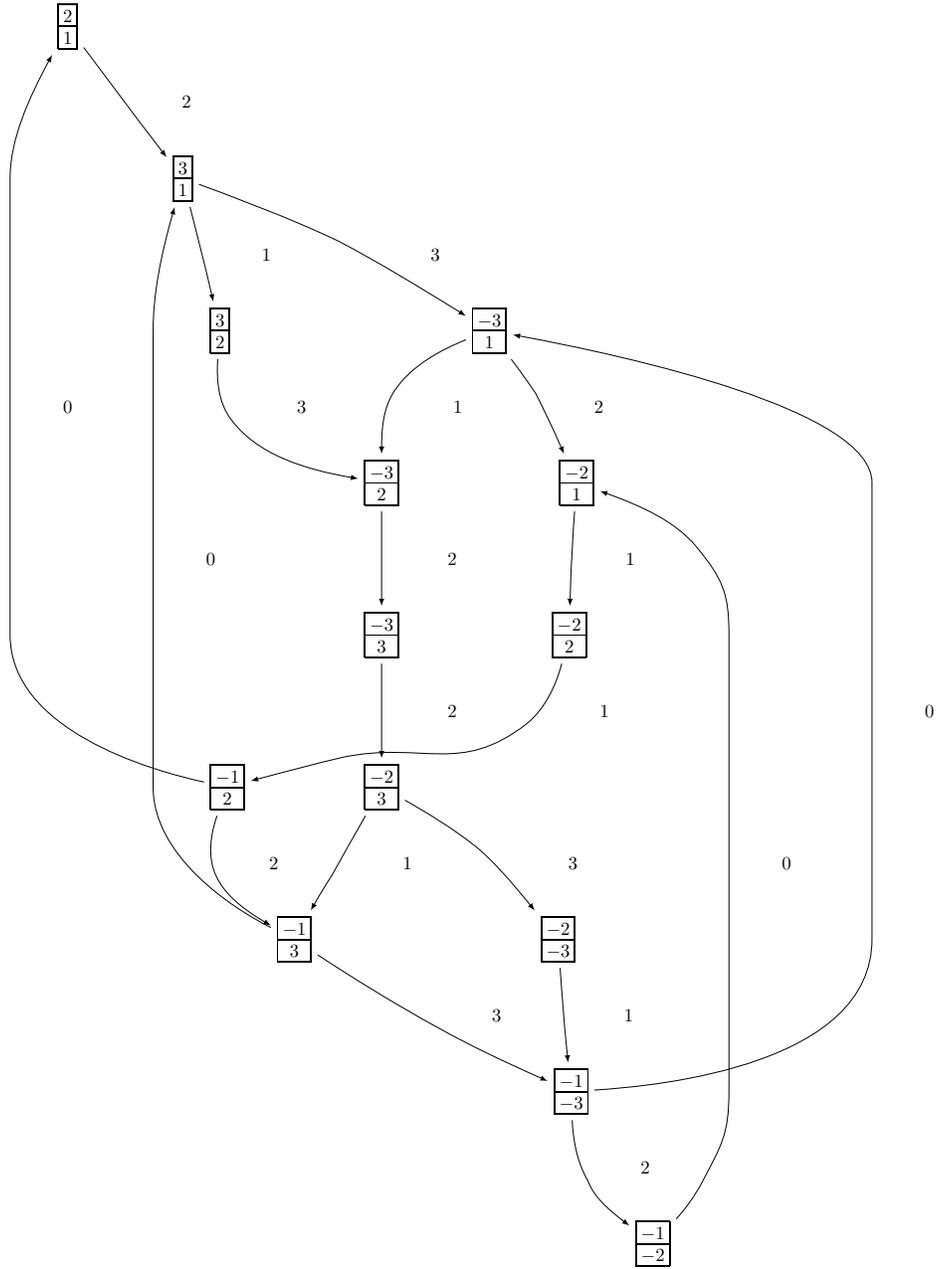
\begin{figure}
  \scalebox{.7}{
    \begin{tikzpicture}[>=latex,join=bevel,]
\node (N_1) at (364bp,16bp) [draw,draw=none] {${\def\lr#1#2#3{\multicolumn{1}{#1@{\hspace{.6ex}}c@{\hspace{.6ex}}#2}{\raisebox{-.3ex}{$#3$}}}\raisebox{-.6ex}{$\begin{array}[b]{c}\cline{1-1}\lr{|}{|}{-1}\\\cline{1-1}\lr{|}{|}{-2}\\\cline{1-1}\end{array}$}}$};
  \node (N_2) at (323bp,426bp) [draw,draw=none] {${\def\lr#1#2#3{\multicolumn{1}{#1@{\hspace{.6ex}}c@{\hspace{.6ex}}#2}{\raisebox{-.3ex}{$#3$}}}\raisebox{-.6ex}{$\begin{array}[b]{c}\cline{1-1}\lr{|}{|}{-2}\\\cline{1-1}\lr{|}{|}{1}\\\cline{1-1}\end{array}$}}$};
  \node (N_3) at (135bp,262bp) [draw,draw=none] {${\def\lr#1#2#3{\multicolumn{1}{#1@{\hspace{.6ex}}c@{\hspace{.6ex}}#2}{\raisebox{-.3ex}{$#3$}}}\raisebox{-.6ex}{$\begin{array}[b]{c}\cline{1-1}\lr{|}{|}{-1}\\\cline{1-1}\lr{|}{|}{2}\\\cline{1-1}\end{array}$}}$};
  \node (N_4) at (49bp,672bp) [draw,draw=none] {${\def\lr#1#2#3{\multicolumn{1}{#1@{\hspace{.6ex}}c@{\hspace{.6ex}}#2}{\raisebox{-.3ex}{$#3$}}}\raisebox{-.6ex}{$\begin{array}[b]{c}\cline{1-1}\lr{|}{|}{2}\\\cline{1-1}\lr{|}{|}{1}\\\cline{1-1}\end{array}$}}$};
  \node (N_5) at (171bp,180bp) [draw,draw=none] {${\def\lr#1#2#3{\multicolumn{1}{#1@{\hspace{.6ex}}c@{\hspace{.6ex}}#2}{\raisebox{-.3ex}{$#3$}}}\raisebox{-.6ex}{$\begin{array}[b]{c}\cline{1-1}\lr{|}{|}{-1}\\\cline{1-1}\lr{|}{|}{3}\\\cline{1-1}\end{array}$}}$};
  \node (N_6) at (319bp,344bp) [draw,draw=none] {${\def\lr#1#2#3{\multicolumn{1}{#1@{\hspace{.6ex}}c@{\hspace{.6ex}}#2}{\raisebox{-.3ex}{$#3$}}}\raisebox{-.6ex}{$\begin{array}[b]{c}\cline{1-1}\lr{|}{|}{-2}\\\cline{1-1}\lr{|}{|}{2}\\\cline{1-1}\end{array}$}}$};
  \node (N_7) at (111bp,590bp) [draw,draw=none] {${\def\lr#1#2#3{\multicolumn{1}{#1@{\hspace{.6ex}}c@{\hspace{.6ex}}#2}{\raisebox{-.3ex}{$#3$}}}\raisebox{-.6ex}{$\begin{array}[b]{c}\cline{1-1}\lr{|}{|}{3}\\\cline{1-1}\lr{|}{|}{1}\\\cline{1-1}\end{array}$}}$};
  \node (N_8) at (320bp,98bp) [draw,draw=none] {${\def\lr#1#2#3{\multicolumn{1}{#1@{\hspace{.6ex}}c@{\hspace{.6ex}}#2}{\raisebox{-.3ex}{$#3$}}}\raisebox{-.6ex}{$\begin{array}[b]{c}\cline{1-1}\lr{|}{|}{-1}\\\cline{1-1}\lr{|}{|}{-3}\\\cline{1-1}\end{array}$}}$};
  \node (N_9) at (276bp,508bp) [draw,draw=none] {${\def\lr#1#2#3{\multicolumn{1}{#1@{\hspace{.6ex}}c@{\hspace{.6ex}}#2}{\raisebox{-.3ex}{$#3$}}}\raisebox{-.6ex}{$\begin{array}[b]{c}\cline{1-1}\lr{|}{|}{-3}\\\cline{1-1}\lr{|}{|}{1}\\\cline{1-1}\end{array}$}}$};
  \node (N_10) at (218bp,426bp) [draw,draw=none] {${\def\lr#1#2#3{\multicolumn{1}{#1@{\hspace{.6ex}}c@{\hspace{.6ex}}#2}{\raisebox{-.3ex}{$#3$}}}\raisebox{-.6ex}{$\begin{array}[b]{c}\cline{1-1}\lr{|}{|}{-3}\\\cline{1-1}\lr{|}{|}{2}\\\cline{1-1}\end{array}$}}$};
  \node (N_11) at (131bp,508bp) [draw,draw=none] {${\def\lr#1#2#3{\multicolumn{1}{#1@{\hspace{.6ex}}c@{\hspace{.6ex}}#2}{\raisebox{-.3ex}{$#3$}}}\raisebox{-.6ex}{$\begin{array}[b]{c}\cline{1-1}\lr{|}{|}{3}\\\cline{1-1}\lr{|}{|}{2}\\\cline{1-1}\end{array}$}}$};
  \node (N_12) at (313bp,180bp) [draw,draw=none] {${\def\lr#1#2#3{\multicolumn{1}{#1@{\hspace{.6ex}}c@{\hspace{.6ex}}#2}{\raisebox{-.3ex}{$#3$}}}\raisebox{-.6ex}{$\begin{array}[b]{c}\cline{1-1}\lr{|}{|}{-2}\\\cline{1-1}\lr{|}{|}{-3}\\\cline{1-1}\end{array}$}}$};
  \node (N_13) at (218bp,262bp) [draw,draw=none] {${\def\lr#1#2#3{\multicolumn{1}{#1@{\hspace{.6ex}}c@{\hspace{.6ex}}#2}{\raisebox{-.3ex}{$#3$}}}\raisebox{-.6ex}{$\begin{array}[b]{c}\cline{1-1}\lr{|}{|}{-2}\\\cline{1-1}\lr{|}{|}{3}\\\cline{1-1}\end{array}$}}$};
  \node (N_14) at (218bp,344bp) [draw,draw=none] {${\def\lr#1#2#3{\multicolumn{1}{#1@{\hspace{.6ex}}c@{\hspace{.6ex}}#2}{\raisebox{-.3ex}{$#3$}}}\raisebox{-.6ex}{$\begin{array}[b]{c}\cline{1-1}\lr{|}{|}{-3}\\\cline{1-1}\lr{|}{|}{3}\\\cline{1-1}\end{array}$}}$};
  \draw [<-] (N_2) ..controls (359bp,413bp) and (376bp,405bp)  .. (387bp,392bp) .. controls (402bp,374bp) and (405bp,366bp)  .. (405bp,344bp) .. controls (405bp,344bp) and (405bp,344bp)  .. (405bp,98bp) .. controls (405bp,75bp) and (401bp,69bp)  .. (391bp,50bp) .. controls (387bp,42bp) and (381bp,34bp)  .. (N_1);
  \pgfsetstrokecolor{black}
  \draw (436bp,221bp) node {$0$};
  \draw [<-] (N_4) ..controls (28bp,633bp) and (18bp,611bp)  .. (18bp,590bp) .. controls (18bp,590bp) and (18bp,590bp)  .. (18bp,344bp) .. controls (18bp,293bp) and (93bp,271bp)  .. (N_3);
  \draw (49bp,467bp) node {$0$};
  \draw [->] (N_3) ..controls (126bp,237bp) and (124bp,224bp)  .. (129bp,214bp) .. controls (133bp,205bp) and (142bp,197bp)  .. (N_5);
  \draw (160bp,221bp) node {$2$};
  \draw [->] (N_2) ..controls (321bp,398bp) and (320bp,383bp)  .. (N_6);
  \draw (352bp,385bp) node {$1$};
  \draw [->] (N_4) ..controls (67bp,648bp) and (84bp,625bp)  .. (N_7);
  \draw (113bp,631bp) node {$2$};
  \draw [->] (N_8) ..controls (321bp,73bp) and (323bp,60bp)  .. (329bp,50bp) .. controls (332bp,42bp) and (338bp,36bp)  .. (N_1);
  \draw (360bp,57bp) node {$2$};
  \draw [<-] (N_9) ..controls (352bp,494bp) and (482bp,465bp)  .. (482bp,426bp) .. controls (482bp,426bp) and (482bp,426bp)  .. (482bp,180bp) .. controls (482bp,113bp) and (370bp,101bp)  .. (N_8);
  \draw (513bp,303bp) node {$0$};
  \draw [->] (N_5) ..controls (198bp,162bp) and (225bp,145bp)  .. (249bp,132bp) .. controls (265bp,123bp) and (284bp,114bp)  .. (N_8);
  \draw (280bp,139bp) node {$3$};
  \draw [<-] (N_7) ..controls (99bp,548bp) and (95bp,527bp)  .. (95bp,508bp) .. controls (95bp,508bp) and (95bp,508bp)  .. (95bp,262bp) .. controls (95bp,225bp) and (137bp,197bp)  .. (N_5);
  \draw (126bp,385bp) node {$0$};
  \draw [->] (N_6) ..controls (312bp,317bp) and (306bp,304bp)  .. (296bp,296bp) .. controls (260bp,268bp) and (240bp,287bp)  .. (197bp,278bp) .. controls (184bp,275bp) and (169bp,271bp)  .. (N_3);
  \draw (338bp,303bp) node {$1$};
  \draw [->] (N_9) ..controls (292bp,487bp) and (297bp,480bp)  .. (301bp,474bp) .. controls (305bp,466bp) and (309bp,458bp)  .. (N_2);
  \draw (335bp,467bp) node {$2$};
  \draw [->] (N_9) ..controls (251bp,498bp) and (233bp,489bp)  .. (224bp,474bp) .. controls (220bp,467bp) and (218bp,459bp)  .. (N_10);
  \draw (259bp,467bp) node {$1$};
  \draw [->] (N_7) ..controls (135bp,581bp) and (168bp,569bp)  .. (195bp,556bp) .. controls (216bp,545bp) and (239bp,531bp)  .. (N_9);
  \draw (247bp,549bp) node {$3$};
  \draw [->] (N_7) ..controls (118bp,562bp) and (122bp,547bp)  .. (N_11);
  \draw (156bp,549bp) node {$1$};
  \draw [->] (N_12) ..controls (315bp,152bp) and (316bp,137bp)  .. (N_8);
  \draw (351bp,139bp) node {$1$};
  \draw [->] (N_13) ..controls (204bp,237bp) and (197bp,225bp)  .. (191bp,214bp) .. controls (189bp,211bp) and (187bp,208bp)  .. (N_5);
  \draw (232bp,221bp) node {$1$};
  \draw [->] (N_13) ..controls (241bp,249bp) and (258bp,239bp)  .. (271bp,228bp) .. controls (279bp,221bp) and (287bp,212bp)  .. (N_12);
  \draw (321bp,221bp) node {$3$};
  \draw [->] (N_10) ..controls (218bp,398bp) and (218bp,383bp)  .. (N_14);
  \draw (256bp,385bp) node {$2$};
  \draw [->] (N_11) ..controls (129bp,482bp) and (130bp,469bp)  .. (137bp,460bp) .. controls (151bp,441bp) and (177bp,433bp)  .. (N_10);
  \draw (175bp,467bp) node {$3$};
  \draw [->] (N_14) ..controls (218bp,316bp) and (218bp,301bp)  .. (N_13);
  \draw (256bp,303bp) node {$2$};
\end{tikzpicture}}
\caption{$B^{2,1}$ of type $C_3^{(1)}$ \label{fig:B21}}
\end{figure}
In this section we present the affine crystal structure for $B^{2,2}$ and $B^{2,1}$ of type $C_3^{(1)}$.
We also list all minimal elements for $B^{2,3}$ of type $C_3^{(1)}$ to illustrate that $\ve$ is not
a bijection and hence $B^{2,3}$ is not perfect.

\subsection{KR crystal $B^{2,2}$} 
The KR crystal $B^{2,2}$ has three classcial components
\begin{equation*}
	B^{2,2} \cong B(2\La_2) \oplus B(2\La_1) \oplus B(0).
\end{equation*}
The unique element in $B(0)$ is denoted by $u$. Since $f_0$ commutes with 
$f_2,f_3$ and the classical $C_3$-crystal structure is explicitly known by ~\cite{KN:1994},
it suffices to determine $f_0$ on each $\{2,3\}$-component. All $\{2,3\}$-highest weight crystal
elements are given in Table~\ref{tab:ex} together with the action of $f_0$.

\begin{table}
\begin{equation*}
\begin{array}{|l|l|}
\hline b & f_0(b) \\ \hline \mbox{}&\\[-3mm]
	\young(22,11) & \young(22)\\[3mm]
 	\young(23,12) & \young(3\ab,22)\\[0.3cm]
	\young(33,22) & \emptyset\\[3mm]
	\young(2\bb,12) & \young(2\ab)\\[3mm]
	\young(2\ab,12) & \emptyset \\[3mm]
	\young(3\ab,12) & \emptyset \\[3mm]
	\young(3\ab,22) & \emptyset \\[3mm]
	\young(\bb\ab,12) & \emptyset \\[3mm]
	\hline
\end{array}
\qquad
\begin{array}{|l|l|}
\hline b & f_0(b) \\ \hline \mbox{} & \\[-3mm]
	\young(\bb\ab,22) & \emptyset \\[3mm]
	\young(\ab\ab,22) & \emptyset \\[3mm]
	\young(12) & \young(\bb\ab,22) \\[3mm]
	\young(22) & \young(\ab\ab,22) \\[3mm]
	\young(1\ab) & \emptyset \\[3mm]
	\young(2\ab) & \emptyset \\[3mm]
	\young(\ab\ab) & \emptyset \\[3mm]
	\young(11) & u \\[3mm]
	u & \young(\ab\ab)\\[3mm] \hline
\end{array}
\end{equation*}
\caption{Action of $f_0$ on $\{2,3\}$-highest weight elements in $B^{2,2}$ of type $C_3^{(1)}$
\label{tab:ex}}
\end{table}

The bijection $\ve: B^{2,2}_{\min} \to P_1^+$ is given by
\begin{equation*}
\begin{array}{|l|l|} \hline
	b & \ve(b)\\ \hline
	u & \La_0\\[3mm] 
	\young(1\ab) & \La_1\\[3mm]
	\young(2\ab,1\bb) & \La_2\\[3mm] 
	\young(3\bb,2\cb) & \La_3\\[3mm] \hline
\end{array}
\end{equation*}

\subsection{KR crystal $B^{2,1}$}
The KR crystal graph for $B^{2,1}$ of type $C_3^{(1)}$ is given in Figure~\ref{fig:B21}. It has only 
one classical component
\begin{equation*}
	B^{2,1} \cong B(\La_2).
\end{equation*}
$B^{2,1}$ is not perfect, since $\ve$ is not a bijection from minimal elements to level 1 dominant
weights:
\begin{equation*}
\begin{array}{|l|l|} \hline
	b & \ve(b)\\ \hline 	
	\mbox{} & \\[-3mm]
	\young(2,1) & \La_0\\[3mm]
 	\young(3,2) \quad \young(\bb,2) & \La_1\\[3mm]
	\young(\cb,3) \quad \young(\ab,\bb) & \La_2 \\[3mm]
	\young(\bb,\cb) & \La_3 \\[3mm] \hline
\end{array}
\end{equation*}

\subsection{KR crystal $B^{2,3}$}
The KR crystal $B^{2,3}$ of type $C_3^{(1)}$ is also not perfect. The map $\ve$ from the minimal
elements to level 2 dominant weights is given below:
\begin{equation*}
\begin{array}{|l|l|} \hline
	b & \ve(b)\\ \hline 	
	\mbox{} & \\[-3mm]
	\young(2,1) & 2\La_0\\[4mm]
	\young(2,11\ab) \quad \young(\bb,2) \quad \young(3,2) & \La_0+\La_1\\[4mm]
	\young(22\ab,11\bb) \quad \young(\ab,\bb) \quad \young(\cb,3) \quad \young(\bb,12\bb)
	\quad \young(3,12\bb) & \La_0+\La_2\\[4mm]
	\young(23\bb,12\cb) \quad \young(\bb,\cb) \quad \young(\cb,13\cb) & \La_0+\La_3\\[4mm]	
	\young(3,12\ab) \quad \young(\bb,12\ab) \quad \young(2,1\ab\ab) & 2\La_1\\[4mm]
	\young(23\ab,12\bb) \quad \young(2\bb\ab,12\bb) \quad \young(\cb,13\ab) \quad
	\young(\bb,2\bb\ab) \quad \young(3,2\bb\ab) & \La_1+\La_2\\[4mm]
	\young(3\cb\ab,13\cb) \quad \young(33\bb,22\cb) \quad \young(2\cb\ab,13\cb) \quad
	\young(\bb,1\cb\ab) \quad \young(\cb,3\cb\ab) & \La_1+\La_3\\[4mm]
	\young(2\ab\ab,1\bb\bb) \quad \young(2\cb\ab,13\bb) \quad 
	\young(3\cb\ab,13\bb) & 2\La_2\\[4mm]
	\young(3\bb\ab,2\cb\bb) \quad \young(3\cb\bb,23\cb) \quad
	\young(2\bb\ab,1\cb\bb) & \La_2+\La_3\\[4mm]
	\young(3\bb\bb,2\cb\cb) & 2\La_3\\[4mm] \hline
	\end{array}
\end{equation*}
Under the embedding $S:B^{2,3} \to B^{2,3}_{A_7^{(2)}}$ of Definition~\ref{def:KR C}
we have
\begin{equation*}
	S\left( \young(\bb,2\bb\ab) \right) = \young(2\bb\ab,1\cb\bb) = b_1 \quad \text{and} \quad
	S\left( \young(\cb,13\ab) \right) = \young(2\db\ab,14\bb) = b_2
\end{equation*}
which are precisely the two elements $b_1,b_2$ of Proposition~\ref{prop:C nonperfect} such that
$\ve(b_1)=\ve(b_2)=\La_2+\La_3$ in type $A_7^{(2)}$.

\end{document}